\documentclass[reqno,oneside]{amsart}
\usepackage{hyperref}
\usepackage{geometry}
\usepackage[ansinew]{inputenc}
\usepackage{graphicx}
\usepackage{amsmath}
\usepackage{amsthm}
\usepackage{amssymb, color}%
\usepackage[numbers, square]{natbib}
\usepackage{mathrsfs}
\usepackage{bbm}
\usepackage{tikz}
\usepackage{mathptmx}

\newcommand{\mmp}{\mathbb{P}}

\newcommand{\tp}{\overset{{\rm P}}{\to}}

\newcommand{\me}{\mathbb{E}}
\newcommand{\E}{\mathbb{E}}

\newcommand{\mr}{\mathbb{R}}
\newcommand{\mn}{\mathbb{N}}

\DeclareMathOperator{\1}{\mathbbm{1}}

\newtheorem{thm}{Theorem}[section]
\newtheorem{lemma}[thm]{Lemma}

\newtheorem{assertion}[thm]{Proposition}
\theoremstyle{definition}

\theoremstyle{remark}
\newtheorem{rem}[thm]{Remark}

\begin{document}

\title[On intermediate levels of nested occupancy scheme]{On intermediate levels of nested occupancy scheme in random environment generated by stick-breaking II}

\author{Alexander Iksanov}
\address{Faculty of Computer Science and Cybernetics, Taras Shevchenko National University of Kyiv, Kyiv, Ukraine}
\email{iksan@univ.kiev.ua}

\author{Alexander Marynych}
\address{Faculty of Computer Science and Cybernetics, Taras Shevchenko National University of Kyiv, Kyiv, Ukraine}
\email{marynych@unicyb.kiev.ua}

\author{Igor Samoilenko}
\address{Faculty of Computer Science and Cybernetics, Taras Shevchenko National University of Kyiv, 01601 Kyiv, Ukraine}
\email{isamoil@i.ua}

\begin{abstract}
A nested occupancy scheme in random environment is a generalization of the classical Karlin infinite balls-in-boxes occupancy scheme in random environment (with random probabilities). Unlike the Karlin scheme in which the collection of boxes is unique, there is a nested hierarchy of boxes, and the hitting probabilities of boxes are defined in terms of iterated fragmentation of a unit mass. In the present paper we assume that the random fragmentation law is given by stick-breaking in which case the infinite occupancy scheme defined by the first level boxes is known as the Bernoulli sieve. Assuming that $n$ balls have been thrown, denote by $K_n(j)$ the number of occupied boxes in the $j$th level and call the level $j$ intermediate if $j=j_n\to\infty$ and $j_n=o(\log n)$ as $n\to\infty$. We prove a multidimensional central limit theorem for the vector $(K_n(\lfloor j_n u_1\rfloor),\ldots, K_n(\lfloor j_n u_\ell\rfloor)$, properly normalized and centered, as $n\to\infty$, where $j_n\to\infty$ and $j_n=o((\log n)^{1/2})$. The present paper continues the line of investigation initiated in the article \cite{Buraczewski+Dovgay+Iksanov:2020} in which the occupancy of intermediate levels $j_n\to\infty$, $j_n=o((\log n)^{1/3})$ was analyzed.
\end{abstract}



\maketitle

\section{Introduction}\label{Sect1}

A nested occupancy scheme in random environment can be thought of as the infinite Karlin occupancy scheme \cite{Karlin:1967} in random environment (with random probabilities) settled in a consistent way on the tree of a weighted branching process. The scheme and its multitype counterpart were investigated in \cite{Bertoin:2008, Buraczewski+Dovgay+Iksanov:2020, Gnedin+Iksanov:2020, Joseph:2011} and \cite{Businger:2017}, respectively. The present work is a continuation of the paper \cite{Buraczewski+Dovgay+Iksanov:2020}. Therefore, we only recall briefly the definition of the scheme and refer to the cited papers for a much more detailed description, motivation and examples.

Let $P:=(P_r)_{r\in\mn}$ be a collection of positive random variables with an arbitrary joint distribution satisfying $\sum_{r\geq 1}P_r=1$ a.s. Also, let $\mathbb{V}=\cup_{n\in\mn_0}\mn^n$ be the set of all possible individuals of some population, where $\mn_0:=\mn\cup\{0\}$. The ancestor is identified with
the empty word $\varnothing$ and its weight is $P(\varnothing)=1$.
On some probability space $(\Omega, \mathcal{F}, \mmp)$ let
$((P_r(v))_{r\in\mn})_{v \in \mathbb{V}}$ be a family of
independent copies of $(P_r)_{r\in\mn}$. An individual $v
=v_1\ldots v_j$ of the $j$th generation whose weight is denoted by
$P(v)$ produces an infinite number of offspring residing in the $(j+1)$th generation. The
offspring of the individual $v$ are enumerated by $vr = v_1 \ldots
v_j r$, where $r\in \mn$, and the weights of the offspring are
denoted by $P(vr)$. It is postulated that $P(vr)=P(v)P_r(v)$.
Observe that, for each $j\in\mn$, $\sum_{|v|=j} P(v)=1$ a.s.,
where, by convention, $|v|=j$ means that the sum is taken over all
individuals of the $j$th generation. For $j\in\mn$, denote by
$\mathcal{F}_j$ the $\sigma$-algebra generated by
$(P(v))_{|v|=1},\ldots, (P(v))_{|v|=j}$. The nested sequence of
environments is formed by the weights of the subsequent
generations individuals, that is, $(P(v))_{|v|=1}$, $(P(v))_{|v|=2},\ldots$. Further, we identify individuals with
`boxes'. At time $j=0$, all `balls' are collected in the box
$\varnothing$. At time $j=1$, given $\mathcal{F}_1$, `balls' are
allocated independently with probability $P(v)$ of hitting box
$v$, $|v|=1$. At time $j=k$, given $\mathcal{F}_k$, a ball located
in the box $v$ with $|v|=k$ is placed independently of the others
into the box $vr$ with probability $P_r(v)=P(vr)/P(v)$.

Of course, restricting attention to the $j$th generation we obtain the Karlin occupancy scheme with random probabilities $(P(v))_{|v|=j}$. The quantity of traditional interest is the number of occupied boxes in the $j$th generation that we denote by $K_n(j)$ when $n$ balls have been thrown. Since, by construction, the probabilities in subsequent generations are `nested' we infer that the sequence $K_n(j)$ is a.s.\ nondecreasing and thus eventually reaches the value $n$. It is natural to call $\tau_n:=\inf\{j\in\mn: K_n(j)=n\}$ the height of the scheme or the extinction time, for $K_n(j)=n$ for $j\geq \tau_n$. By Theorem 1 in \cite{Joseph:2011},
\begin{equation}\label{height}
\lim_{n\to\infty}(\tau_n/\log n)={\rm const}\quad \text{a.s.},
\end{equation}
where ${\rm const}$ is the explicitly given positive constant.

In the following we assume that the random probabilities $P$ are given by {\it stick-breaking} (a.k.a.\ a residual allocation model)
\begin{equation}\label{BS}
P_r:=W_1W_2\cdot\ldots\cdot W_{r-1}(1-W_r),\quad r\in\mn,
\end{equation}
where $W_1$, $W_2,\ldots$ are independent copies of a random variable $W$ taking values in $(0,1)$. To decide whether `balls' fall into the given `box' of the $1$st generation one may design a sieving procedure involving Bernoulli random variables with random parameters. This explains the term `Bernoulli sieve' used for the occupancy scheme in the $1$st generation. The Bernoulli sieve was introduced in \cite{Gnedin:2004} and then investigated in many articles, an incomplete list of recent contributions includes \cite{Alsmeyer+Iksanov+Marynych:2017, Duchamps+Pitman+Tang:2017+,  Iksanov:2016, Iksanov+Jedidi+Bouzeffour:2017, Pitman+Tang:2017+}. Recall that if $\mmp\{W\in {\rm d}x\}=\theta x^{\theta-1}\1_{(0,1)}(x){\rm d}x$ for some $\theta>0$ the distribution of $P$ is called ${\rm GEM}$ distribution (after Griffiths-Engen-McCloskey) with parameter $\theta$.

We call the $j$th generation {\it early}, {\it intermediate} or {\it late} depending on whether $j$ is fixed, $j=j_n\to\infty$ and $j_n=o(\log n)$ as $n\to\infty$, or $j$ is of order $\log n$. In view of \eqref{height} there are no other generations. It was shown in \cite{Bertoin:2008} that in the late generations $j$ $K_n(j)$, properly normalized, converges a.s.\ to a deterministic multiple of the terminal value $Z$, say, of the Biggins martingale. Since the terminal value is a function of the boundary of the weighted branching process tree, Bertoin's result reveals that the late generations belong to the range of boundary dominance. On the other hand, it was shown in \cite{Gnedin+Iksanov:2020} and, under \eqref{BS}, in \cite{Buraczewski+Dovgay+Iksanov:2020} that in the early generations and the intermediate generations with $j=j_n=o((\log n)^{1/3})$, $K_n(j)$, properly normalized and centered, exhibits Gaussian fluctuations inherited partially from the fluctuations of $K_n(1)$. It is important that the variable $Z$ does not pop up in these results. Thus, the aforementioned generations belong to the range of root dominance. The purpose of the present article is to show that the intermediate generations $j_n=o((\log n)^{1/2})$ belong to the range of root dominance. More precisely, a consequence of Theorem \ref{main100} which is our main result is that, under \eqref{BS} and some additional assumptions on $W$, $K_n(j_n)$ satisfies a central limit theorem similar to that proved in \cite{Buraczewski+Dovgay+Iksanov:2020} for the range $j_n=o((\log n)^{1/3})$ but with a more complicated centering. On the one hand, the latter fact settles in part the conjecture stated on p.~6 in the last cited paper. On the other hand, the need for a more complicated centering leads to significant technical complications. Last but not least, unlike in \cite{Buraczewski+Dovgay+Iksanov:2020} we do not assume that the distribution of $|\log W|$ is nonarithmetic.

We close the introduction by explaining which new effects can be expected in the analysis of $K_n(j_n)$ beyond the level $j_n=o((\log n)^{1/2})$, that is, for intermediate generations growing like $(\log n)^{1/2}$ or faster. Our preliminary calculations make it highly plausible that $K_n(j_n)$ exhibits a countable number of phase transitions at the levels $j_n\sim \lambda (\log n)^{1-1/r}$ for $r=2,3,\ldots$ and $\lambda>0$. To provide a more quantitative justification of the aforementioned phase transitions we state a consequence of our main result
\begin{equation}\label{eq:wlln_kn_jn}
\frac{{j_n ! (\me |\log W|)^{j_n}}K_n(j_n)}{(\log n)^{j_n}}~\tp 1,\quad n\to\infty\quad\text{whenever}\quad j_n=o((\log n)^{1/2})
\end{equation} 
which holds true under the assumptions $\me (\log W)^2<\infty$ and $\me |\log(1-W)|<\infty$. We conjecture that, for integer $r\geq 3$, under the assumptions $\me |\log W|^r<\infty$ and $\me |\log(1-W)|^{r-1}<\infty$, 
\begin{equation*}
\exp{\Big(-\sum_{k=2}^{r-1}(\gamma_k j_n^k)/(\log n)^{k-1}\Big)}\frac{{j_n ! (\me |\log W|)^{j_n}}K_n(j_n)}{(\log n)^{j_n}}~\tp 1,\quad n\to\infty\quad\text{whenever}\quad
j_n=o((\log n)^{1-1/r})
\end{equation*}
for appropriate constants $\gamma_2,\ldots, \gamma_{r-1}$. We stress that even derivation of a weak law of large numbers, let alone distributional limit theorems, at the levels beyond $j_n=o((\log n)^{1/2})$ requires techniques that are completely different from those exploited in the present paper.

\section{Main results}\label{main345}

We write  $\Rightarrow$, ${\overset{{\rm
d}}\longrightarrow}$ and ${\overset{{\rm f.d.d.}}\longrightarrow}$
to denote weak convergence in a function space, weak convergence
of one-dimensional and finite-dimensional distributions,
respectively. For $j\in\mn$ and $t>0$ put $\rho_j(t):=\#\{u\in\mathbb{V}:|u|=j,\,P(u)\geq 1/t\}$. Thus, $\rho_j$ is the counting function for the probabilities in the $j$th generation. Here is our main result.
\begin{thm}\label{main100}
Assume that $P$ is given by \eqref{BS}, that $\sigma^2:={\rm Var}(\log W)\in (0,\infty)$ and that $\me |\log(1-W)|<\infty$. Let $(j_n)_{n\in\mn}$ be a sequence of positive numbers satisfying
$j_n\to \infty$ and $j_n=o((\log n)^{1/2})$ as $n\to\infty$. The following
limit theorem holds, as $n\to\infty$,
\begin{equation}\label{clt5555}
\Bigg(\frac{\lfloor j_n\rfloor^{1/2}(\lfloor j_n u\rfloor-1)!\big(K_n(\lfloor j_n u\rfloor)-\me \rho_{\lfloor j_n u\rfloor}(n)\big)}{(\sigma^2 \mu^{-2\lfloor j_n u\rfloor-1} (\log n)^{2\lfloor j_n
u\rfloor-1})^{1/2}}\Bigg)_{u>0}~{\overset{{\rm f.d.d.}}\longrightarrow}~\Bigg(\int_{[0,\,\infty)}e^{-uy}{\rm
d}B(y)\Bigg)_{u>0},
\end{equation}
where $(B(v))_{v\geq 0}$ is a standard Brownian motion and $\mu:=\me |\log W|<\infty$.
\end{thm}
\begin{rem}
Under the additional assumption that the distribution of $|\log W|$ is nonlattice, a version of Theorem \ref{main100} was proved in \cite{Buraczewski+Dovgay+Iksanov:2020} for the range
$j=j_n\to \infty$ and $j_n=o((\log n)^{1/3})$ as $n\to\infty$. Actually, it was shown there that for such $j$ the centering $\me \rho_{\lfloor j_n u\rfloor}(n)$ can be replaced with $(\log n)^{\lfloor j_n u\rfloor}/((\lfloor j_n u\rfloor)!\mu^{\lfloor j_n u\rfloor})$. Such a replacement is not always possible for $j$ growing faster (yet slower than $(\log n)^{1/2}$), as the next result shows.
\end{rem}
\begin{assertion}\label{precise}
There exist distributions of $W$ satisfying the assumptions of Theorem \ref{main100} for which $$\lim_{n\to\infty}\frac{j_n^{1/2}(j_n-1)!|\me \rho_{j_n}(n)-(\log n)^{j_n}/((j_n)!\mu^{j_n})|}{\mu^{-j_n}(\log n)^{j_n-1/2}}
$$ is finite and positive provided that $j_n\sim a(\log n)^{1/3}$ as $n\to\infty$ for some $a>0$ and is infinite provided that $\lim_{n\to\infty}(\log n)^{-1/3}j_n=\infty$ and $j_n=o((\log n)^{1/2})$ as $n\to\infty$.
\end{assertion}

\section{Limit theorems for a branching random walk}

As in \cite{Buraczewski+Dovgay+Iksanov:2020} we shall work with a branching random walk which is an additive counterpart of the original weighted branching process obtained by the logarithmic transformation.

Let $(\xi_i, \eta_i)_{i\in\mn}$ be independent copies of a random vector $(\xi, \eta)$ with positive arbitrarily dependent components. Denote by $(S_i)_{i\geq 0}$ the zero-delayed
ordinary random walk with increments $\xi_i$ for $i\in\mn$, that is, $S_0:=0$ and $S_i:=\xi_1+\cdots+\xi_i$ for $i\in\mn$. Define
\begin{equation*}
T_i:=S_{i-1}+\eta_i,\quad i\in \mn.
\end{equation*}
The sequence $T:=(T_i)_{i\in\mn}$ is called {\it perturbed random walk}. The recent book \cite{Iksanov:2016} provides a good overview of known results for the so defined perturbed random walks.
Put $N(t):=\sum_{i\geq 1}\1_{\{T_i\leq t\}}$ and $V(t):=\me N(t)$ for $t\geq 0$. One can check that
\begin{equation}\label{equ}
V(t)=\me U((t-\eta)^+)=\int_{[0,\,t]}U(t-y){\rm d}G(y),\quad t\geq 0
\end{equation}
where, for $t\geq 0$, $U(t):=\sum_{i\geq 0}\mmp\{S_i\leq t\}$ is
the renewal function and $G(t)=\mmp\{\eta\leq t\}$. In what follows we shall write $k_1\ast k_2$ 
for the Lebesgue--Stieltjes convolution of 
functions $k_1$ and $k_2$ 
of locally bounded variation. Thus,
$$
(k_1 \ast k_2)(t) = \int_{\mr}k_1(t-y){\rm d}k_2(y),\quad t\in\mr.
$$
Using this notation we can restate \eqref{equ} as $V=U\ast G$.

Here is an informal description of the construction of a branching random walk in the special case it is generated by $T$. At time $0$ there is one individual, the ancestor. The ancestor produces offspring (the first generation) with positions on $[0,\infty)$ given by the points of $T$. The
first generation produces the second generation. The displacements of positions of the second generation individuals with respect to their mothers' positions are distributed according to
copies of $T$, and for different mothers these copies are independent. The second generation produces the third one, and so on. All individuals act independently of each other.

For $t\geq 0$ and $j\in\mn$, denote by $N_j(t)$ the number of the $j$th generation individuals with positions $\leq t$ and put $V_j(t):=\me N_j(t)$. Then $N_1(t)=N(t)$, $V_1(t)=V(t)$ and
$$
V_j(t):=(V_{j-1}\ast V)(t)=\int_{[0,\,t]} V_{j-1}(t-y){\rm d}V(y),\quad j\geq 2,\quad t\geq 0.
$$
We now provide a decomposition of $N_j$ which is of major importance for what follows
\begin{equation}\label{basic1232}
N_j(t)=\sum_{r\geq 1}N^{(r)}_{j-1}(t-T_r)\1_{\{T_r\leq t\}},\quad j\geq 2,\quad t\geq 0.
\end{equation}
Here, $N_{j-1}^{(r)}(t)$ is the number of successors in the $j$th generation which reside in the interval $[T_r,t+T_r]$ of the first generation individual with position $T_r$. By the branching property, $(N_{j-1}^{(1)}(t))_{t\geq 0}$, $(N_{j-1}^{(2)}(t))_{t\geq 0},\ldots$ are independent copies of $(N_{j-1}(t))_{t\geq 0}$ which are also independent of $T$.

Theorem \ref{main100} will be deduced from Theorems \ref{main4} and \ref{main5} when setting $(\xi,\eta)=(|\log W|, |\log(1-W)|)$ and an additional result.
\begin{thm}\label{main4}
Let $t\mapsto j(t)$ be any positive function satisfying $\lim_{t\to\infty} j(t)=\infty$ and $j(t)=o(t^{1/2})$ as $t\to\infty$. Assume that ${\tt s}^2={\rm Var}\,\xi \in (0,\infty)$ and $\me \eta<\infty$. Then, as $t\to\infty$, $$\bigg(\frac{\lfloor j(t)\rfloor^{1/2}(\lfloor j(t)u\rfloor-1)!}{{\tt m}^{-\lfloor j(t)u\rfloor} t^{\lfloor j(t)u\rfloor-1/2}}\Big(N_{\lfloor j(t)u\rfloor}(t)-\sum_{r\geq 1}V_{\lfloor j(t)u\rfloor-1}(t-T_r)\1_{\{T_r\leq t\}}\bigg)_{u>0}~{\overset{{\rm f.d.d.}}\longrightarrow}~(\Theta(u))_{u>0},$$ where $\Theta(u)=0$ for $u>0$ and ${\tt m}:=\me\xi<\infty$. 
\end{thm}

\begin{thm}\label{main5}
Assume that ${\tt s}^2={\rm Var}\,\xi\in (0,\infty)$ and that $\me \eta<\infty$. Let $t\mapsto j(t)$ be any positive function satisfying $j(t)\to \infty$
and $j(t)=o(t^{1/2})$ as $t\to\infty$.  Then, as $t\to\infty$,
\begin{multline}\label{clt22}
\left(\frac{\lfloor j(t)\rfloor^{1/2}(\lfloor j(t)u\rfloor-1)!}{({\tt s}^2{\tt m}^{-2\lfloor j(t)u\rfloor-1}t^{2\lfloor j(t)u\rfloor-1})^{1/2}}\bigg(\sum_{r\geq
1}V_{\lfloor j(t)u\rfloor-1}(t-T_r)\1_{\{T_r\leq t\}}-V_{\lfloor j(t)u \rfloor}(t)\bigg)\right)_{u>0}\\
{\overset{{\rm f.d.d.}}\longrightarrow}~ \Bigg(\int_{[0,\,\infty)}e^{-uy}{\rm d}B(y)\Bigg)_{u>0},
\end{multline}
where $(B(v))_{v\geq 0}$ is a standard Brownian motion.
\end{thm}

The remainder of the paper is organized as follows. Some technical results are stated and proved in Section \ref{auxstat}. 
Theorem \ref{main4} is then deduced from these in Section \ref{main4pro}. Section \ref{sect:flt} is devoted to proving Theorem \ref{main5}. The proofs of Theorems \ref{main100} and Proposition \ref{precise} are given in Sections \ref{sect:main} and \ref{sect:preci}, respectively.

\section{Auxiliary tools}\label{auxstat}

According to Lorden's inequality we have
\begin{equation}\label{lord}
U(t)-{\tt m}^{-1}t \leq c_0,\quad t\geq 0
\end{equation}
for appropriate constant $c_0>0$ whenever $\E\xi^2<\infty$. 
This implies
\begin{equation}\label{lord1}
V(t)-{\tt m}^{-1}t \leq c_0,\quad t\geq 0
\end{equation}
as a consequence of $V(t)\leq U(t)$ for $t\geq 0$. Assume now that $\me\eta<\infty$. Then
\begin{eqnarray}\label{ineq100}
V(t)-{\tt m}^{-1}t&=&\int_{[0,\,t]}(U(t-y)-{\tt m}^{-1}(t-y)){\rm d}G(y)\notag\\&\hphantom{==}-& {\tt m}^{-1} \int_0^t (1-G(y)){\rm d}y\geq-{\tt m}^{-1}\int_0^t (1-G(y)){\rm d}y\geq
-{\tt m}^{-1}\me\eta
\end{eqnarray}
because, with $\nu(t):=\inf\{k\in\mn: S_k>t\}$ for $t\geq 0$, $U(t)={\tt m}^{-1}\E S_{\nu(t)}\geq {\tt m}^{-1}t$. Here, the equality is just Wald's identity. Summarizing, we have shown that
\begin{equation}\label{lord2}
|V(t)-{\tt m}^{-1}t|\leq c,\quad t\geq 0
\end{equation}
where $c=\max(c_0, {\tt m}^{-1}\me\eta)$. Similarly,
\begin{equation}\label{lord20}
0\leq V(t)-{\tt m}^{-1}\int_0^tG(y){\rm d}y\leq c_0,\quad t\geq 0
\end{equation}
as a consequence of $$0\leq V(t)-{\tt m}^{-1}\int_0^tG(y){\rm d}y=\int_{[0,\,t]}(U(t-y)-{\tt m}^{-1}(t-y)){\rm d}G(y)\leq c_0 G(t)\leq c_0.$$

Proposition \ref{aux5000} is borrowed, for the most part, from Proposition 4.1 in \cite{Buraczewski+Dovgay+Iksanov:2020}. Formula \eqref{impo} which is the only new ingredient follows easily from \eqref{ineq2}.
\begin{assertion}\label{aux5000}
Under the assumptions $\me\xi^2<\infty$ and $\me\eta<\infty$,
\begin{equation}\label{ineq2}
\bigg|V_j(t)-\frac{t^j}{j!{\tt m}^j}\bigg|\leq \sum_{i=0}^{j-1}\binom{j}{i}\frac{c^{j-i}t^i}{i!{\tt m}^i},\quad
j\in\mn,~t\geq 0,
\end{equation}
where ${\tt m}=\me\xi<\infty$ and $c$ is the same as in \eqref{lord2}.

In particular, whenever $j=j(t)\to\infty$ and $j(t)=o(t^{1/2})$ as $t\to\infty$,
\begin{equation}\label{basic}
V_j(t)=\frac{t^j}{j!{\tt m}^j}+O\Big(\frac{jt^{j-1}}{(j-1)!{\tt m}^{j-1}}\Big),\quad t\to\infty
\end{equation}
and
\begin{equation}\label{basic1}
\frac{jt^{j-1}}{(j-1)!{\tt m}^{j-1}}=o\Big(\frac{t^j}{j!{\tt m}^j}\Big),
\end{equation}
whence
\begin{equation}\label{auxma}
V_j(t)~\sim~ \frac{t^j}{j!{\tt m}^j},\quad t\to\infty.
\end{equation}
Furthermore, for all $a>0$,
\begin{equation}\label{impo}
\lim_{t\to\infty}\sup_{y\geq at}\,\Big|\frac{V_j(y)j!{\tt m}^j}{y^j}-1\Big|=0.
\end{equation}
\end{assertion}
Asymptotic results collected in Lemma \ref{estim} are extensively used in various proofs, the proof of Theorem \ref{main4} being the most active consumer.
\begin{lemma}\label{estim}
Let $j\in\mn$ and $s\geq 0$ satisfy $s\geq 2c{\tt m}j^2$. Then, for $1\leq k\leq j$,
\begin{equation}\label{vk}
V_k(s)\leq \frac{2s^k}{k!{\tt m}^k},
\end{equation}
\begin{equation}\label{basic123}
\sum_{i=0}^{k-1}\binom{k}{i}\frac{c^{k-i}s^i}{i!{\tt m}^i}\leq \frac{2cks^{k-1}}{(k-1)!{\tt m}^{k-1}}.
\end{equation}
and
\begin{equation}\label{basic12312}
\sum_{i=0}^{k-1}\binom{k}{i}\frac{c^{k-i}s^{i+1}}{(i+1)!{\tt m}^{i+1}}\leq \frac{2cs^k}{(k-1)!{\tt m}^k}.
\end{equation}
Also, whenever $j=j(t)=o(t^{1/2})$ as $t\to\infty$, for any constant $a>0$,
\begin{equation}\label{inter4}
\lim_{t\to\infty}\frac{(j-1)!{\tt m}^{j-1}}{jt^{j-1}}\sum_{i=0}^{j-2}\binom{j}{i}\frac{a^{j-i}t^i}{i!{\tt m}^i}=0\quad\text{and}\quad \lim_{t\to\infty}\frac{(j-1)!{\tt m}^{j-1}}{jt^{j-1}}(j-1)\sum_{i=0}^{j-2}\binom{j-2}{i}\frac{a^{j-2-i}t^i}{i!{\tt m}^i}=0.
\end{equation}
\end{lemma}
\begin{proof}
In view of \eqref{ineq2}, to prove \eqref{vk} it suffices to show that
$$
\sum_{i=0}^{k-1}\binom{k}{i}\frac{c^{k-i}s^i}{i!{\tt m}^i}\leq \frac{s^k}{k!{\tt m}^k},\quad 1\leq k\leq j,\quad s\geq 2c{\tt m}j^2.
$$
Using
\begin{equation}\label{eq:bin}
\binom{k}{i}\leq \frac{k!}{i!}\leq k^{k-i}
\end{equation}
this follows from
\begin{equation}\label{aux1}
\frac{k!{\tt m}^k}{s^k}\sum_{i=0}^{k-1}\binom{k}{i}\frac{c^{k-i}s^i}{i!{\tt m}^i}\leq \sum_{i=0}^{k-1}\left(\frac{c{\tt m}k^2}{s}\right)^{k-i} \leq \sum_{i=0}^{k-1}\left(\frac{c{\tt m}k^2}{2c{\tt m}j^2}\right)^{k-i}=\sum_{i=1}^{k}\left(\frac{k^2}{2j^2}\right)^{i}\leq \sum_{i=1}^{\infty}2^{-i}=1,
\end{equation}
because $k\leq j$.

While inequality \eqref{basic123} was proved in Proposition 4.1 of \cite{Buraczewski+Dovgay+Iksanov:2020}, the proof of inequality \eqref{basic12312} is analogous.

We only check the first part of \eqref{inter4}, the proof of the other part being similar. Invoking \eqref{eq:bin} we obtain $$\frac{(j-1)!{\tt m}^{j-1}}{jt^{j-1}}\sum_{i=0}^{j-2}\binom{j}{i}\frac{a^{j-i}t^i}{i!{\tt m}^i}\leq a\sum_{i=0}^{j-2}\Big(\frac{aj^2}{{\tt m}t}\Big)^{j-1-i}\leq
\frac{(aj)^2}{{\tt m}t}\Big(1-\frac{aj^2}{{\tt m}t}\Big)^{-1}~\to 0,\quad t\to\infty.$$
\end{proof}

Our next technical result is needed for the proof of Theorem \ref{main5}.
\begin{lemma}\label{auxmain5}
Assume that $\me\xi^2<\infty$ and $\me\eta<\infty$. Let $j=j(t)\to\infty$ and $j(t)=o(t^{1/2})$ as $t\to\infty$. Then
$$\lim_{T\to\infty}\limsup_{t\to\infty}\frac{j^{1/2}(j-1)!{\tt m}^j}{t^{j-1/2}}\int_{(Tt/j,\,t]} y^{1/2}{\rm d}_y(-V_{j-1}(t-y))=0.$$
\end{lemma}
\begin{proof}
Integration by parts yields $$\int_{(Tt/j,\,t]} y^{1/2}{\rm d}_y(-V_{j-1}(t-y))=V_{j-1}(t(1-T/j))(Tt/j)^{1/2}+2^{-1}\int_{Tt/j}^t V_{j-1}(t-y) y^{-1/2}{\rm d}y.$$ By \eqref{impo}, $$\frac{j^{1/2}(j-1)!{\tt m}^j}{t^{j-1/2}}V_{j-1}(t(1-T/j))(Tt/j)^{1/2}~\sim~ {\tt m}T^{1/2}(1-T/j)^{j-1}~\to~{\tt m}T^{1/2}e^{-T},\quad t\to\infty.$$ The right-hand side converges to $0$ as $T\to\infty$. Using \eqref{ineq2} we obtain
\begin{multline*}
\frac{j^{1/2}(j-1)!{\tt m}^j}{t^{j-1/2}} \int_{Tt/j}^t V_{j-1}(t-y)y^{-1/2}{\rm d}y\leq
\frac{j^{1/2}{\tt m}}{t^{j-1/2}}\int_{Tt/j}^t (t-y)^{j-1} y^{-1/2}{\rm d}y\\+\frac{j^{1/2}(j-1)!{\tt m}^j}{t^{j-1/2}}\sum_{i=0}^{j-2}\binom{j-1}{i}\frac{c^{j-1-i}}{i!{\tt m}^i}
\int_{Tt/j}^t (t-y)^i y^{-1/2}{\rm d}y =:A_j(t)+B_j(t).
\end{multline*}
Further, by Lebesgue's dominated convergence theorem,
$$A_j(t)={\tt m}\int_T^j (1-y/j)^{j-1} y^{-1/2}{\rm d}y~\to~ {\tt m}\int_T^\infty e^{-y}y^{-1/2}{\rm d}y,\quad t\to\infty,$$ and the right-hand side converges to $0$ as $T\to\infty$. It suffices to show that $\lim_{t\to\infty}\,B_j(t)=0$. 
This is done as follows:
\begin{multline*}
B_j(t)\leq T^{-1/2}\frac{j(j-1)!{\tt m}^j}{t^j}\sum_{i=0}^{j-2}\binom{j-1}{i}\frac{c^{j-1-i}t^{i+1}}{(i+1)!{\tt m}^i}(1-T/j)^{i+1}\leq T^{-1/2}{\tt m}\sum_{i=0}^{j-2}\Big(\frac{c{\tt m}j^2}{t}\Big)^{j-1-i}\\ \leq \frac{c{\tt m}^2j^2}{T^{1/2}t}\Big(1-\frac{c{\tt m}j^2}{t}\Big)^{-1}~\to~0,\quad t\to\infty.
\end{multline*}
We have used \eqref{eq:bin} for the second inequality.
\end{proof}

Lemmas \ref{key1} and \ref{key2} provide `light versions' of the key renewal theorem for the intermediate generations.
\begin{lemma}\label{key1}
Let $f: [0,\infty)\to [0,\infty)$ be a directly Riemann integrable (dRi) function on $[0,\infty)$ and $j=j(t)\to\infty$, $j(t)=o(t^{1/2})$ as $t\to\infty$. Assume that ${\tt m}<\infty$.
Then
\begin{equation}\label{equival}
(f\ast V_{j})(t)=\int_{[0,\,t]}f(t-y){\rm d}V_j(y)=O(V_{j-1}(t))=O\Big(\frac{t^{j-1}}{(j-1)!{\tt m}^j}\Big),\quad t\to\infty.
\end{equation}
\end{lemma}
\begin{proof}
For $t\geq 0$, put $g(t):=\int_{[0,\,t]}f(t-y){\rm d}V(y)$. Then, for $t\geq 0$,
$$
(f\ast V_{j})(t)=(f\ast (V\ast V_{j-1}))(t)=((f\ast V)\ast V_{j-1})(t)=(g\ast V_{j-1})(t)=\int_{[0,\,t]}g(t-y){\rm d}V_{j-1}(y).
$$
By Lemma \ref{key}, $g(t)\leq a$ for some $a>0$ and all $t\geq 0$, and the first asymptotic relation in \eqref{equival} follows. The second is a consequence of \eqref{impo}.
\end{proof}

\begin{lemma}\label{key2}
Let $j=j(t)\to\infty$ and $j(t)=o(t^{1/2})$ as $t\to\infty$ and assume that ${\tt m}<\infty$. Then
\begin{equation*}
\int_{(t,\,\infty)}e^{t-y}{\rm d}V_j(y)=O(V_{j-1}(t))=O\Big(\frac{t^{j-1}}{(j-1)!{\tt m}^j}\Big),\quad t\to\infty.
\end{equation*}
\end{lemma}
\begin{proof}
Let $h:\mr \to [0,\infty)$ be a dRi function on $(-\infty, 0]$ satisfying $h(t)=0$ for $t>0$. We start as in the proof of Lemma \ref{key1}: for $t\geq 0$,
\begin{multline*}
\int_{(t,\,\infty)}h(t-y){\rm d}V_j(y)=(h\ast V_j)(t)=((h\ast V)\ast V_{j-1})(t)=\int_{\mr}(h\ast V)(t-y){\rm d}V_{j-1}(y)\\
=\int_{[0,\,t]}h_1(t-y){\rm d}V_{j-1}(y)+\int_{(t,\,\infty)}h_2(t-y){\rm d}V_{j-1}(y),
\end{multline*}
where $h_1(t):=\int_{(t,\,\infty)}h(t-y){\rm d}V(y)$ 
and $h_2(t):=\int_{[0,\,\infty)}h(t-y){\rm d}V(y)$ for $t\in\mr$. 
By Lemma \ref{key}, $h_1(t)\leq b$ for some $b>0$ and all $t\geq 0$, whence
$$\int_{[0,\,t]}h_1(t-y){\rm d}V_{j-1}(y)=O(V_{j-1}(t)),\quad t\to\infty.$$ From now on we put $h(t)=e^t\1_{(-\infty, 0]}(t)$ and note that the previous centered formula holds true for such $h$.
Further, setting $\rho:=\me e^{-\eta}(1-\me e^{-\xi})^{-1}$ we obtain $h_2(t)=e^t\int_{[0,\,\infty)}e^{-y}{\rm d}V(y)=\rho e^t$ for $t\leq 0$, whence $$\int_{(t,\,\infty)}h_2(t-y){\rm d}V_{j-1}(y)=\rho \int_{(t,\,\infty)}e^{t-y}{\rm d}V_{j-1}(y)=-\rho V_{j-1}(t)+\rho \int_t^\infty e^{t-y}V_{j-1}(y){\rm d}y=:\rho C_j(t)$$ for $t\geq 0$. In view of \eqref{impo}, given $\varepsilon>0$ we have for large enough $t$
$$0\leq C_j(t)\leq -V_{j-1}(t)+(1+\varepsilon)\int_t^\infty e^{t-y}\frac{y^{j-1}}{(j-1)!{\tt m}^{j-1}}{\rm d}y=-V_{j-1}(t)+\frac{1+\varepsilon}{{\tt m}^{j-1}}\sum_{i=0}^{j-1}\frac{t^i}{i!}.$$ Using $$V_{j-1}(t)~\sim~ \frac{t^{j-1}}{(j-1)!{\tt m}^{j-1}},\quad t\to\infty$$ (see \eqref{impo}) and $$\sum_{i=0}^{j-1}\frac{t^i}{i!}~\sim~ \frac{t^{j-1}}{(j-1)!},\quad t\to\infty$$ and sending first $t\to\infty$ and then $\varepsilon\to 0+$ we infer $$\lim_{t\to\infty}\frac{C_j(t)}{V_{j-1}(t)}=0.$$ The proof of Lemma \ref{key2} is complete.
\end{proof}

\section{Proof of Theorem \ref{main4}}\label{main4pro}

The proof is a simplified version of the proof of Lemma 4.3 in \cite{Buraczewski+Dovgay+Iksanov:2020}. The improvement consists in using asymptotic relations rather than precise formulae.

For $j\in\mn$ and $t\geq 0$, put $D_j(t):={\rm Var}\, N_j(t)$ and $$I_j(t):=\me\bigg(\sum_{r\geq 1}V_{j-1}(t-T_r)\1_{\{T_r\leq t\}}-V_j(t)\bigg)^2$$ with the convention that $V_0(t)=1$ for $t\geq 0$. Our starting point is the recursive formula which is a consequence of \eqref{basic1232}: for $j\geq 2$ and $t\geq 0$,
\begin{eqnarray}\label{aux5}
D_j(t)&=&\me \bigg(\sum_{r\geq 1}\big(N^{(r)}_{j-1}(t-T_r)- V_{j-1}(t-T_r)\big)\1_{\{T_r\leq t\}}\bigg)^2\\&+& \me\bigg(\sum_{r\geq 1}V_{j-1}(t-T_r)
\1_{\{T_r\leq t\}}-V_j(t)\bigg)^2=\int_{[0,\,t]}D_{j-1}(t-y){\rm d}V(y)+I_j(t).\notag
\end{eqnarray}
Starting with $D_1(t)=I_1(t)$ and iterating \eqref{aux5}  we obtain
\begin{equation}\label{recur}
\int_{[0,\,t]}D_{j-1}(t-y){\rm d}V(y)=\sum_{k=1}^{j-1}\int_{[0,\,t]}I_k(t-y){\rm d}V_{j-k}(y),\quad j\geq 2, t\geq 0.
\end{equation}
We intend to prove that whenever $j=j(t)\to\infty$ and $j(t)=o(t^{1/2})$ as $t\to\infty$,
\begin{equation}\label{asymp}
\int_{[0,\,t]}D_{j-1}(t-y){\rm d}V(y)= O\Big(\frac{t^{2j-2}}{(j-2)!(j-1)!{\tt m}^{2j-2}}\Big),\quad t\to\infty.
\end{equation}
To this end we use a two-step procedure. First, we show that $I_j$ is bounded from above by a nonnegative and nondecreasing function $h_j$, say, and that the corresponding inequality is valid for all nonnegative arguments. This leads by virtue of \eqref{recur} to a useful inequality for $D_j$ which holds for all nonnegative arguments. Second, we obtain an upper bound for $h_j$ (and for $D_j$) which is valid for large arguments.

\noindent {\sc Step 1}. Throughout this step all formulae hold true for any $j\in\mn$ and all $t\geq 0$ (we do not indicate this explicitly).

We start with
\begin{align*}
&\hspace{-1cm}\E \sum_{r\geq 2} \sum_{1\leq i<r}V_{j-1}(t-T_i)\1_{\{T_i\leq t\}}V_{j-1}(t-T_r)\1_{\{T_r\leq t\}}\\
&\leq \E \sum_{i\geq 1}\me\big(V_{j-1}(t-T_i)\1_{\{T_i\leq t\}}\big(V_{j-1}(t-\eta_{i+1}-S_i)\1_{\{\eta_{i+1}\leq t-S_i\}}\\
&+V_{j-1}(t-\eta_{i+2}-\xi_{i+1}-S_i)\1_{\{\eta_{i+2}+\xi_{i+1}\leq t-S_i\}}+\ldots\big)|(\xi_k,\eta_k)_{1\leq i\leq k}\big)\1_{\{S_i\leq t\}}\\
&=\E \sum_{i\geq 1} V_{j-1}(t-T_i)\1_{\{T_i\leq t\}} V_j(t-S_i)\1_{\{S_i\leq t\}}\leq \E \sum_{i\geq 0} V_{j-1}(t-S_i)V_j(t-S_i)\1_{\{S_i\leq t\}}.
\end{align*}
Hence,
\begin{multline}\label{259}
I_j(t)=\me \sum_{r\geq 1}V^2_{j-1}(t-T_r)\1_{\{T_r\leq t\}}+2\me\sum_{r\geq 2}\sum_{1\leq i<r} V_{j-1}(t-T_i)\1_{\{T_i\leq t\}}V_{j-1}(t-T_r)\1_{\{T_r\leq t\}}-V^2_j(t)\\
\leq V_{j-1}(t)V_j(t)+2\int_{[0,\,t]}V_{j-1}(t-y)V_j(t-y){\rm d}U(y)- V^2_j(t).
\end{multline}
Put $\tilde U(t):=\sum_{i\geq 1}\mmp\{S_i\leq t\}$ for $t\geq 0$ and note that, in view of \eqref{lord} and Wald's identity, $$|\tilde U(t)-{\tt m}^{-1}t|\leq \tilde c:=c_0+1.$$ With this at hand integration by parts yields
\begin{multline*}
\int_{[0,\,t]}V_{j-1}(t-y)V_j(t-y){\rm d}U(y)=V_{j-1}(t)V_j(t)+\int_{[0,\,t]}V_{j-1}(t-y)V_j(t-y){\rm d}\tilde U(y)\\\leq (\tilde c+1)V_{j-1}(t)V_j(t)+{\tt m}^{-1}\int_0^t V_{j-1}(y)V_j(y){\rm d}y,
\end{multline*}
whence, by \eqref{259}, $$I_j(t)\leq (2\tilde c+3) V_{j-1}(t)V_j(t)+2{\tt m}^{-1}\int_0^t V_{j-1}(y)V_j(y){\rm d}y- V^2_j(t).$$ Invoking \eqref{ineq2} we obtain
\begin{align}
2{\tt m}^{-1}\int_0^t V_{j-1}(y)V_j(y){\rm d}y&\leq 2{\tt m}^{-1}\int_0^t \Big(\frac{y^{j-1}}{(j-1)!{\tt m}^{j-1}}+\sum_{i=0}^{j-2}\binom{j-1}{i}\frac{c^{j-1-i}y^i}{i!{\tt m}^i}\Big)\Big(\frac{y^j}{j!{\tt m}^j}+\sum_{i=0}^{j-1}\binom{j}{i}\frac{c^{j-i}y^i}{i!{\tt m}^i}\Big){\rm d}y\notag\\
&\leq\frac{t^{2j}}{(j!)^2{\tt m}^{2j}}+2\sum_{i=0}^{j-2}\binom{j-1}{i}\frac{c^{j-1-i}t^{j+1+i}}{(j+1+i)j!i!{\tt m}^{j+1+i}}+2 \sum_{i=0}^{j-1}\binom{j}{i}\frac{c^{j-i}t^{j+i}}{(j+i)(j-1)!i!{\tt m}^{j+i}}\notag\\
&+ 2\Big( \sum_{i=0}^{j-2}\binom{j-1}{i}\frac{c^{j-1-i}t^i}{i!{\tt m}^i}\Big)\int_0^t \sum_{i=0}^{j-1}\binom{j}{i}\frac{c^{j-i}y^i}{i!{\tt m}^{i+1}}{\rm d}y\notag\\
&=:\frac{t^{2j}}{(j!)^2{\tt m}^{2j}}+f_j(t)
\label{deffj}
\end{align}
Another appeal to \eqref{ineq2} yields 
\begin{multline*}
V^2_j(t)-\frac{t^{2j}}{(j!)^2{\tt m}^{2j}}=\Big(V_j(t)+\frac{t^j}{j!{\tt m}^j}\Big)\Big(V_j(t)-\frac{t^j}{j!{\tt m}^j}\Big)\geq -\Big(V_j(t)+\frac{t^j}{j!{\tt m}^j}\Big)\sum_{i=0}^{j-1}\binom{j}{i}\frac{c^{j-i}t^i}{i!{\tt m}^i}=:-g_j(t).
\end{multline*}
It is important for what follows that both $f_j$ and $g_j$ are nonnegative nondecreasing functions. Combining pieces together we infer
\begin{equation}\label{defhj}
I_j(t)\leq 
(2\tilde c+3) V_{j-1}(t)V_j(t)+f_j(t)+g_j(t)=:h_j(t) 
\end{equation}
and thereupon
$$D_{j-1}(t)=\sum_{k=1}^{j-1}\int_{[0,\,t]} I_k(t-y){\rm d}V_{j-k-1}(y)\leq h_{j-1}(t) 
+\sum_{k=1}^{j-2}h_k(t) 
V_{j-k-1}(t),\quad j\geq 2,\quad t\geq 0,$$ for $h_j$ is a nondecreasing function.

\noindent {\sc Step 2}. Fix now $j\in\mn$ and $s\geq 0$ satisfying $s\geq 2c{\tt m}j^2$ and let $1\leq k\leq j$. Throughout this step we tacitly assume that all formulae hold true for this range of parameters. By \eqref{vk}, $$V_{k-1}(s)V_k(s)\leq \frac{4s^{2k-1}}{(k-1)!k!{\tt m}^{2k-1}}\leq \frac{4s^{2k-1}}{((k-1)!)^2{\tt m}^{2k-1}}.$$ 
Our next task is to show that $$f_k(s)\leq \frac{12cs^{2k-1}}{((k-1)!)^2{\tt m}^{2k-1}}.$$ While the second and third summands in the definition of $f_k$ (see \eqref{deffj}) are estimated with the help of \eqref{basic123}, the fourth summand is dealt with as follows: 
\begin{multline*}
2\Big( \sum_{i=0}^{k-2}\binom{k-1}{i}\frac{c^{k-1-i}s^i}{i!{\tt m}^i}\Big)\Big(\sum_{i=0}^{k-1}\binom{k}{i}\frac{c^{k-i}s^{i+1}}{(i+1)!{\tt m}^{i+1}}\Big)\\\leq 2\frac{2c(k-1)s^{k-2}}{(k-2)!{\tt m}^{k-2}}\frac{2cs^k}{(k-1)!{\tt m}^k}=\frac{8c^2 s^{2k-2}}{((k-2)!)^2{\tt m}^{2k-2}}\leq \frac{4c s^{2k-1}}{((k-1)!)^2{\tt m}^{2k-1}}.
\end{multline*}
Here, we have used \eqref{basic123} and \eqref{basic12312} to bound the first and second factor, respectively, and the inequality $s\geq 2c{\tt m}(k-1)^2$ for the last passage. Finally, $$g_k(s)\leq \frac{6cs^{2k-1}}{((k-1)!)^2 {\tt m}^{2k-1}}$$ by \eqref{vk} and \eqref{basic123}. Summarizing, we have shown that
\begin{equation}\label{ik}
h_k(s)\leq \frac{As^{2k-1}}{((k-1)!)^2 {\tt m}^{2k-1}},
\end{equation}
where $A:=12+8 \tilde c+18c$.

Further, we obtain, for $s\geq 2\max (c,1){\tt m}j^2=:a_j$,
\begin{align}
D_{j-1}(s)&\leq h_{j-1}(s)+\sum_{k=1}^{j-2}h_k(s)V_{j-k-1}(s)\notag\\
&\leq \frac{As^{2j-3}}{((j-2)!)^2 {\tt m}^{2j-3}}+2A\sum_{k=1}^{j-2}\frac{s^{j-2+k}}{(j-k-1)!((k-1)!)^2 {\tt m}^{j-2+k}}\notag\\
&= \frac{As^{2j-3}}{((j-2)!)^2 {\tt m}^{2j-3}}\Big(1+2 \sum_{k=1}^{j-2}\binom{j-2}{k-1}\frac{(j-2)!}{(k-1)!}\Big(\frac{{\tt m}}{s}\Big)^{j-k-1}\Big)\notag\\
&\leq \frac{As^{2j-3}}{((j-2)!)^2 {\tt m}^{2j-3}}\Big(1+2\frac{{\tt m} j^2}{s}\Big(1-\frac{{\tt m} j^2}{s}\Big)^{-1}  \Big)\leq \frac{3As^{2j-3}}{((j-2)!)^2 {\tt m}^{2j-3}}.\label{dj}
\end{align}
Here, the first inequality is just formula \eqref{defhj}, the second inequality follows from \eqref{vk} and \eqref{ik}, and the penultimate inequality is a consequence of \eqref{eq:bin}.

We are ready to prove relation \eqref{asymp}. To this end, let $j=j(t)\to\infty$ and $j(t)=o(t^{1/2})$ as $t\to\infty$. Noting that the inequality $t\geq a_j$ holds then true for large enough $t$ we obtain
\begin{align*}
\int_{[0,\,t]}D_{j-1}(t-y){\rm d}V(y)&=\int_{[0,\,t-a_j]} D_{j-1}(t-y){\rm d}V(y)+\int_{(t-a_j,\,t]} D_{j-1}(t-y){\rm d}V(y)\\
&\leq \frac{3A}{((j-2)!)^2 {\tt m}^{2j-3}}\int_{[0,\,t]}(t-y)^{2j-3}{\rm d}V(y)+ \left(\max_{s\in [0,\,a_j]}D_{j-1}(s)\right)
U(a_j)\\
&\leq \frac{3At^{2j-2}}{2(j-2)!(j-1)!{\tt m}^{2j-2}}+\frac{3Act^{2j-3}}{((j-2)!)^2 {\tt m}^{2j-3}}+\left(\max_{s\in [0,\,a_j]}D_{j-1}(s)\right)U(a_j)
\end{align*}
having utilized \eqref{dj} and \eqref{subad} for the first inequality and integration by parts together with \eqref{lord2} for the second. 
It is clear that $$\frac{3Act^{2j-3}}{((j-2)!)^2 {\tt m}^{2j-3}}=o\Big(\frac{t^{2j-2}}{(j-2)!(j-1)!{\tt m}^{2j-2}}\Big),\quad t\to\infty.$$ Further,
$$
\max_{s\in [0,\,a_j]}D_{j-1}(s)U(a_j)\leq (D_{j-1}(a_j)+V^2_{j-1}(a_j))U(a_j)\leq \Big(\frac{3Aa_j^{2j-3}}{((j-2)!)^2 {\tt m}^{2j-3}}+\frac{4a_j^{2j-2}}{((j-1)!)^2 {\tt m}^{2j-2}}\Big)U(a_j),
$$
where the second inequality follows from \eqref{vk} and \eqref{dj}. By the elementary renewal theorem, with $a=2\max(c,1){\tt m}$,
$$\frac{(j-2)!(j-1)!{\tt m}^{2j-2}}{t^{2j-2}}\frac{a_j^{2j-2}}{((j-1)!)^2 {\tt m}^{2j-2}} U(a_j)~\sim~\Big(\frac{aj^2}{t}\Big)^{2j-2}\frac{aj}{{\tt m}}~\to~ 0,\quad t\to\infty$$ because $(aj^2/t)^{2j-2}$ converges to $0$ faster than any negative power of $j$. Similarly, $$\frac{(j-2)!(j-1)!{\tt m}^{2j-2}}{t^{2j-2}}\frac{a_j^{2j-3}}{((j-2)!)^2 {\tt m}^{2j-3}} U(a_j)~\sim~\Big(\frac{aj^2}{t}\Big)^{2j-2}j~\to~ 0,\quad t\to\infty.$$ Thus, $$\max_{s\in [0,\,a_j]}D_{j-1}(s)U(a_j)=o\Big(\frac{t^{2j-2}}{(j-2)!(j-1)!{\tt m}^{2j-2}}\Big),$$ and \eqref{asymp} follows.

In view of Markov's inequality and the Cram\'{e}r-Wold device weak convergence of the finite-dimensional distributions to the zero vector in Theorem \ref{main4} follows from
$$\lim_{t\to\infty}\frac{\lfloor j(t)\rfloor ((\lfloor j(t)u\rfloor-1)!)^2}{{\tt m}^{-2\lfloor j(t)u\rfloor} t^{2\lfloor j(t)u\rfloor-1}}\me
\Big(N_{\lfloor j(t)u\rfloor}(t)-\sum_{r\geq 1}V_{\lfloor j(t)u\rfloor-1}(t-T_r)\1_{\{T_r\leq t\}}\Big)^2=0$$ for each $u>0$ which is ensured by \eqref{aux5} and \eqref{asymp}. The proof of Theorem \ref{main4} is complete.

\section{Proof of Theorem \ref{main5}}\label{sect:flt}

We start with a couple of lemmas. As usual, we denote by $D$ the Skorokhod space of right-continuous functions defined on $[0,\infty)$ with finite limits from the left at positive points.
\begin{lemma}\label{lem:flt}
Under the assumptions and notation of Theorem \ref{main5}, as $t\to\infty$,
\begin{equation}\label{eq:flt}
\Big(\frac{N(ut)-V(ut)}{({\tt s}^2 {\tt m}^{-3}t)^{1/2}}\Big)_{u\geq 0}~\Rightarrow~ (B(u))_{u\geq 0}
\end{equation}
in the $J_1$-topology on $D$.
\end{lemma}
\begin{proof}
According to part (B1) of Theorem 3.2 in \cite{Alsmeyer+Iksanov+Marynych:2017}
\begin{equation}\label{aux99}
\Big(\frac{N(ut)-{\tt m}^{-1}\int_0^{ut}G(y){\rm d}y}{({\tt s}^2 {\tt m}^{-3}t)^{1/2}}\Big)_{u\geq 0}~\Rightarrow~ (B(u))_{u\geq 0},\quad t\to\infty.
\end{equation}
In view of \eqref{lord20}, relation \eqref{eq:flt} is an immediate consequence.
\end{proof}
\begin{lemma}\label{aux123}
Under the assumptions and notation of Theorem \ref{main5},
\begin{equation*}\label{230}
\lim_{t\to\infty}t^{-1/2}\me|N(t)-V(t)|={\tt s}^2{\tt m}^{-3}\me |B(1)|.
\end{equation*}
\end{lemma}
\begin{proof}
A specialization of \eqref{aux99} with $u=1$ reads
\begin{equation}\label{aux100}
\frac{N(t)-V(t)}{({\tt s}^2 {\tt m}^{-3}t)^{1/2}}~{\overset{{\rm
d}}\longrightarrow}~ B(1),\quad t\to\infty.
\end{equation}
By Lemma 4.2(b) in \cite{Gnedin+Iksanov:2020}, $\me (N(t)-V(t))^2=O(t)$ as $t\to\infty$ which ensures that the family $(t^{-1/2}(N(t)-V(t)))_{t\geq 1}$ is uniformly integrable. This together with \eqref{aux100} completes the proof.
\end{proof}

The last result that we need is a slight reformulation of Lemma A.5 in \cite{Iksanov:2013}.
\begin{lemma}\label{iks2013}
Let $0\leq a<b<\infty$ and, for each $n\in\mn$, $y_n: [0,\infty)\to [0,\infty)$ be a right-continuous bounded and nondecreasing function. Assume that $\lim_{n\to\infty}x_n=x$ in the $J_1$-topology on $D$ and that, for each $t\geq 0$, $\lim_{n\to\infty}y_n(t)=y(t)$, where $y:[0,\infty)\to [0,\infty)$ is a bounded continuous function. Then
$$\lim_{n\to\infty}\int_{[a,\,b]}x_n(t){\rm d}y_n(t)=\int_{[a,\,b]}x(t){\rm d}y(t).$$
\end{lemma}
\begin{proof}[Proof of Theorem \ref{main5}]
For notational simplicity we shall write $j$ for $j(t)$. According to the Cram\'{e}r-Wold device it suffices to show that for any $\ell\in\mn$, any real $\alpha_1,\ldots, \alpha_\ell$ and any
$0<u_1<\ldots<u_\ell<\infty$, as $t\to\infty$,
\begin{equation}\label{fidi}
\sum_{i=1}^\ell \alpha_i\frac{j^{1/2}(\lfloor j u_i\rfloor-1)! Z(j u_i,
t)}{({\tt s}^2{\tt m}^{-2\lfloor j u_i\rfloor-1}t^{2\lfloor ju_i\rfloor-1})^{1/2}}~{\overset{{\rm
d}}\longrightarrow}~\sum_{i=1}^\ell \alpha_i u_i\int_0^\infty B(y)e^{-u_i y}{\rm d}y,
\end{equation}
where $$Z(ju, t):=\sum_{r\geq 1}V_{\lfloor ju\rfloor-1}(t-T_r)\1_{\{T_r\leq t\}}-V_{\lfloor ju\rfloor}(t).$$

We have for any $u, T>0$ and sufficiently large $t$
\begin{eqnarray*}
\frac{j^{1/2}(\lfloor j u\rfloor-1)! Z(ju,t)}{({\tt s}^2{\tt m}^{-2\lfloor j u\rfloor-1}t^{2\lfloor ju\rfloor-1})^{1/2}}&=&\frac{j^{1/2}(\lfloor j u\rfloor-1)!}{({\tt s}^2{\tt m}^{-3} t^{2\lfloor ju\rfloor-1})^{1/2}}
\int_{[0,\,t]}V_{\lfloor ju\rfloor-1}(t-y){\rm d}(N(y)-V(y))\\&=&\frac{(\lfloor j u\rfloor-1)!{\tt m}^{\lfloor j u\rfloor-1}}{t^{\lfloor ju\rfloor-1}} \int_{[0,\,T]}\frac{N(yt/j)-V(yt/j)}{({\tt s}^2{\tt m}^{-3}t/j)^{1/2}}{\rm d}_y(- V_{\lfloor ju\rfloor-1}(t(1-y/j)))\\&+&\frac{j^{1/2}(\lfloor j u\rfloor-1)!}{({\tt s}^2{\tt m}^{-2\lfloor j u\rfloor-1}t^{2\lfloor ju\rfloor-1})^{1/2}}\int_{(Tt/j,\,t]}(N(y)-V(y)){\rm d}_y(-V_{\lfloor ju\rfloor-1}(t-y)).
\end{eqnarray*}
By Lemma \ref{lem:flt}, $$\Big(\frac{N(ut/j)-V(ut/j)}{({\tt s}^2 {\tt m}^{-3}t/j)^{1/2}}\Big)_{u\geq 0}~\Rightarrow~ (B(u))_{u\geq 0}$$
in the $J_1$-topology on $D$. Here, we have used the assumption $t/j(t)\to\infty$. By Skorokhod's
representation theorem there exist versions $\widehat N_t$ and $\widehat B$ of $((N(ut/j)-V(ut/j))/({\tt s}^2{\tt m}^{-3}t/j)^{1/2})_{u\geq 0}$ and $B$, respectively such that
\begin{equation}\label{cs}
\lim_{t\to\infty}\sup_{y\in [0,\,T]}\Big|\widehat N_t(y)-\widehat B(y)\Big|=0\quad\text{a.s.}
\end{equation}
for all $T>0$. In view of \eqref{impo}, $$\frac{(\lfloor j u\rfloor-1)!{\tt m}^{\lfloor j u\rfloor-1}}{t^{\lfloor ju\rfloor-1}}V_{\lfloor ju\rfloor-1}(t(1-y/j))~\sim~ (1-y/j)^{\lfloor ju\rfloor-1}~\to~e^{-uy},\quad t\to\infty$$ for each fixed $y\geq 0$. By Lemma \ref{iks2013}, this in combination with \eqref{cs}
yields
$$\lim_{t\to\infty}\sum_{i=1}^\ell \alpha_i\frac{(\lfloor j u\rfloor-1)!{\tt m}^{\lfloor j u\rfloor-1}}{t^{\lfloor ju\rfloor-1}}  \int_0^T \widehat N_t(y){\rm d}_y(-V_{\lfloor ju\rfloor-1}(t(1-y/j)))=\sum_{i=1}^\ell \alpha_i
u_i\int_0^T \widehat{B}(y)e^{-u_iy}{\rm d}y\quad \text{a.s.}$$ and
thereupon
$$
\sum_{i=1}^\ell \alpha_i \frac{(\lfloor j u\rfloor-1)!{\tt m}^{\lfloor j u\rfloor-1}}{t^{\lfloor ju\rfloor-1}}\int_{[0,\,T]}\frac{N(yt/j)-V(yt/j)}{({\tt s}^2{\tt m}^{-3}t/j)^{1/2}}{\rm d}_y(-V_{\lfloor ju\rfloor-1}(t(1-y/j)))~{\overset{{\rm
d}}\longrightarrow} ~\sum_{i=1}^\ell \alpha_i u_i\int_0^T B(y)e^{-u_iy}{\rm
d}y,
$$
as $t\to\infty$. Since $\lim_{T\to\infty}\sum_{i=1}^\ell \alpha_i u_i \int_0^T
B(y)e^{-u_iy}{\rm d}y=\sum_{i=1}^\ell \alpha_i u_i \int_0^\infty
B(y)e^{-u_iy}{\rm d}y$ a.s.\ we are left with proving that
$$\lim_{T\to\infty}{\lim\sup}_{t\to\infty}\,\mmp\bigg\{\bigg|\sum_{i=1}^\ell\alpha_i \frac{j^{1/2}(\lfloor j u_i\rfloor-1)!{\tt m}^{\lfloor j u\rfloor}}{t^{\lfloor ju_i\rfloor-1/2}}\int_{(Tt/j,\,t]}(N(y)-V(y)){\rm d}(-V_{\lfloor ju\rfloor-1}(t-y))\bigg|>\varepsilon\bigg\}=0$$ for all $\varepsilon>0$. This limit relation is ensured by Markov's inequality and Lemma \ref{auxmain5} because, by Lemma \ref{aux123}, $\me |N(y)-V(y)| \sim {\tt s}{\tt m}^{-3/2} \me |B(1)|y^{1/2}$ as $y\to\infty$. The proof of Theorem \ref{main5} is complete.
\end{proof}

\section{Proof of Theorem \ref{main100}}\label{sect:main}
Throughout the proof of Theorem \ref{main100} we assume that $T$ is a perturbed random walk generated by $(\xi,\eta)=(|\log W|,|\log(1-W)|)$, so that, for $j\in\mn$ and $n\in\mn$,
$$
\me \rho_j(n)=\me\left(\sum_{|u|=j}\1_{\{|\log P(u)|\leq \log n\}}\right)=V_j(\log n).
$$
We shall use a decomposition
\begin{align*}
K_n(\lfloor j_n u\rfloor)-\me \rho_{\lfloor j_n u\rfloor}(n)&=K_n(\lfloor j_n u\rfloor)-V_{\lfloor j_n u\rfloor}(\log n)= (K_n(\lfloor j_n u\rfloor)-\rho_{\lfloor j_n u\rfloor}(n))\\
&+\Big(\rho_{\lfloor j_n u\rfloor}(n)-\sum_{r\geq 1}V_{\lfloor j_n u\rfloor-1}(\log n-T_r)\1_{\{T_r\leq \log n\}}\Big)\\
&+\sum_{r\geq 1}V_{\lfloor j_n u\rfloor-1}(\log n-T_r)\1_{\{T_r\leq \log n\}}-V_{\lfloor j_n u\rfloor}(\log n)\\
&=:Y_1(n,u)+Y_2(n,u)+Y_3(n,u).
\end{align*}
It suffices to check that, as $n\to\infty$,
\begin{equation}\label{1a}
\Bigg(\frac{\lfloor j_n\rfloor^{1/2}(\lfloor j_n u\rfloor-1)!\,Y_i(n,u)}{\mu^{-\lfloor j_n u\rfloor}(\log n)^{\lfloor j_n u\rfloor-1/2}}\Bigg)_{u>0}~{\overset{{\rm f.d.d.}}\longrightarrow}~(\Theta(u))_{u>0};
\end{equation}
for $i=1,2$ and
\begin{equation}\label{3a}
\Bigg(\frac{\lfloor j_n\rfloor^{1/2}(\lfloor j_n u\rfloor-1)!\,Y_3(n,u)}{(\sigma^2 \mu^{-2\lfloor j_n u\rfloor-1}(\log n)^{2\lfloor j_n
u\rfloor-1})^{1/2}}\Bigg)_{u>0}~{\overset{{\rm f.d.d.}}\longrightarrow}~\Big(\int_{[0,\,\infty)}e^{-uy}{\rm d}B(y)\Big)_{u>0}.
\end{equation}

Formulae \eqref{1a} with $i=2$ and \eqref{3a} follow from Theorems \ref{main4} and \ref{main5}, respectively, in which we replace $t$ with $\log n$ and choose any positive function $t\mapsto j(t)$ satisfying $j(\log n)=j_n$ and $j(t)=o(t^{1/2})$ as $t\to\infty$. Note that ${\tt s}^2=\sigma^2$ and ${\tt m}=\mu$.

In view of Markov's inequality and the Cram\'{e}r-Wold device, relation \eqref{1a} with $i=1$ follows if we can prove that, for fixed $u>0$,
\begin{equation}\label{aux21}
\lim_{n\to\infty} \frac{\lfloor j_n\rfloor^{1/2}(\lfloor j_n u\rfloor-1)!\,\me\big|K_n(\lfloor j_n u\rfloor)-\rho_{\lfloor j_n u\rfloor}(n)\big|}{\mu^{-\lfloor j_n u\rfloor}(\log n)^{\lfloor j_n u\rfloor-1/2}}=0.
\end{equation}
In Section 6 of \cite{Buraczewski+Dovgay+Iksanov:2020}, see the top of page 21, it was shown that, for fixed $u>0$,
\begin{equation*}
\me\big|K_n(\lfloor j_n u\rfloor)-\rho_{\lfloor j_n u\rfloor}(n)\big|\leq n\int_{(n,\,\infty)}x^{-1}{\rm d}\me \rho_{\lfloor j_n u\rfloor}(x)+\int_{[1,\,n]}e^{-n/x}{\rm d}\me \rho_{\lfloor j_n u\rfloor}(x).
\end{equation*}
By Lemma \ref{key2}  applied in the particular setting $(\xi,\eta)=(|\log W|,|\log(1-W)|)$ in which
$V_{\lfloor j_n u\rfloor}(\log x)=\me \rho_{\lfloor j_n u\rfloor}(x)$ for $x\geq 1$ and ${\tt m}=\mu$, we obtain, as $n\to\infty$,
$$
n\int_{(n,\,\infty)}x^{-1}{\rm d}\me \rho_{\lfloor j_n u\rfloor}(x)=\int_{(\log n,\,\infty)}e^{\log n-x}{\rm d}V_{\lfloor j_n u\rfloor}(x)=O(V_{\lfloor j_n u\rfloor-1}(\log n))=O\Big(\frac{(\log n)^{\lfloor j_n u\rfloor-1}}{(\lfloor j_n u\rfloor -1)!\mu^{\lfloor j_n u\rfloor-1}}\Big),
$$
for every fixed $u>0$. In the same vein, applying Lemma \ref{key1} with $f(x)=e^{-e^{x}}$ we have
$$\int_{[1,\,n]}e^{-n/x}{\rm d}\me \rho_{\lfloor j_n u\rfloor}(x)=\int_{[0,\,\log n]}e^{-e^{\log n - x}}{\rm d}V_{\lfloor j_n u\rfloor}(x)=O(V_{\lfloor j_n u\rfloor-1}(\log n))=O\Big(\frac{(\log n)^{\lfloor j_n u\rfloor-1}}{(\lfloor j_n u\rfloor -1)!\mu^{\lfloor j_n u\rfloor-1}}\Big),
$$
for every fixed $u>0$, as $n\to\infty$. In both lemmas we have chosen any positive function $t\mapsto j(t)$ satisfying $j(\log n)=j_n$ and $j(t)=o(t^{1/2})$ as $t\to\infty$. Therefore, formula \eqref{aux21} is justified and this finishes the proof of Theorem \ref{main100}.

\section{Proof of Proposition \ref{precise}}\label{sect:preci}

Proposition \ref{precise} will be deduced from the following result.
\begin{assertion}\label{precise2}
Assume that the distribution of $\xi$ has an absolutely continuous component, $\me e^{\beta_1\xi}<\infty$, $\me e^{\beta_2\eta}<\infty$ for some $\beta_1,\beta_2>0$ and
$$\gamma:=\frac{\me\xi^2}{2{\tt m}^2}-\frac{\me \eta}{{\tt m}}>0.$$ Then
\begin{equation}\label{prec}
V_j(t)-\frac{t^j}{j!{\tt m}^j}~\sim~\frac{\gamma jt^{j-1}}{(j-1)!{\tt m}^{j-1}},\quad t\to\infty
\end{equation}
whenever $j=j(t)=o(t^{1/2})$ as $t\to\infty$ ($j$ is allowed to be fixed).
\end{assertion}
\begin{proof}
We intend to prove that
$$\Big|V_j(t)-\sum_{i=0}^j\binom{j}{i}\frac{\gamma^{j-i}t^i}{i!{\tt m}^i}\Big|\leq \tilde{A}(j-1)\sum_{i=0}^{j-2}\binom{j-2}{i}\frac{c^{j-2-i}t^i}{i!{\tt m}^i},\quad j\geq 2,\quad t\geq 0$$ with the constant $c$ defined in \eqref{lord2} and a constant $\tilde{A}>0$ to be defined below. This together with \eqref{inter4} entails $$\Big|V_j(t)-\frac{t^j}{j!{\tt m}^j}-\frac{\gamma jt^{j-1}}{(j-1)!{\tt m}^{j-1}}\Big|=o\Big(\frac{jt^{j-1}}{(j-1)!{\tt m}^{j-1}}\Big),\quad t\to\infty,$$ whence \eqref{prec}. 

It is convenient to define $V_0(t):=1$ for $t\geq 0$. By Lemma \ref{eq:Borov}, $$V(t)={\tt m}^{-1}t+\gamma+T(t),\quad t\geq 0,$$
where $T$ satisfies \eqref{eq:relat2}, that is, 
\begin{equation*}
|T(t)|\leq C_0e^{-\beta_0 t},\quad t\geq 0
\end{equation*}
for some $C_0>0$ and $\beta_0\in (0,\min(\beta_1,\beta_2))$.

Using this we obtain, for $j\in\mn$ and $t\geq 0$,
$$V_j(t)=\int_{[0,\,t]}V(t-y){\rm d}V_{j-1}(y)={\tt m}^{-1}\int_0^t V_{j-1}(y){\rm d}y+\gamma V_{j-1}(t)+\int_{[0,\,t]}T(t-y){\rm d}V_{j-1}(y).$$ For $j\in\mn$ and $t\geq 0$, put
$$W_j(t):=\sum_{i=0}^j \binom{j}{i} \frac{t^i \gamma^{j-i}}{i!{\tt m}^i}$$ and $T_j(t):=V_j(t)-W_j(t)$. Note that, for $t\geq 0$, $T_0(t)=0$ and $T_1(t)=T(t)$. We shall prove below that
\begin{equation}\label{inter5}
W_j(t)={\tt m}^{-1}\int_0^t W_{j-1}(y){\rm d}y+\gamma W_{j-1}(t),\quad j\in\mn,\quad t\geq 0,
\end{equation}
and thereupon
\begin{equation}\label{recursi}
T_j(t)={\tt m}^{-1}\int_0^t T_{j-1}(y){\rm d}y+\gamma T_{j-1}(t)+\int_{[0,\,t]}T(t-y){\rm d}V_{j-1}(y),\quad j\in\mn,\quad t\geq 0.
\end{equation}
In view of \eqref{eq:relat2} we have
$$
\Big|\int_{[0,\,t]}T(t-y){\rm d}V(y)\Big|\leq C_0\int_{[0,\,t]}e^{-\beta_0(t-y)}{\rm d}V(y)\leq \tilde{A}_1,\quad t\geq 0,
$$
for some $\tilde{A}_1>0$. Here, the last inequality follows from Lemma \ref{key} because the function $t\mapsto C_0e^{-\beta_0 t}$ is dRi on $[0,\infty)$. Thus, by \eqref{recursi}
\begin{equation}\label{inter2}
|T_2(t)|\leq C_0({\tt m}\beta_0)^{-1}+\gamma C_0+\tilde{A}_1=:\tilde{A}_2,\quad t\geq 0,
\end{equation}
and also, see the proof of Lemma \ref{key2} for an accurate argument,
\begin{equation}\label{inter3}
\Big|\int_{[0,\,t]}T(t-y){\rm d}V_i(y)\Big|\leq \tilde{A}_3V_{i-1}(t),\quad i\in\mn,\quad t\geq 0,
\end{equation}
for some $\tilde{A}_3>0$. We shall show with the help of the mathematical induction that recursive equation \eqref{recursi} entails
\begin{equation}\label{inter1}
|T_j(t)|\leq \tilde{A}(j-1)\sum_{i=0}^{j-2}\binom{j-2}{i}\frac{c^{j-2-i}t^i}{i!{\tt m}^i},\quad j\geq 2, \quad t\geq 0,
\end{equation}
where $\tilde{A}:=\max(\tilde{A}_2,\tilde{A}_3)>0$. Observe that this gives the inequality stated at the very beginning of the proof. 
\begin{align*}
|T_k(t)|&\leq {\tt m}^{-1}\int_0^t |T_{k-1}(y)|{\rm d}y+\gamma |T_{k-1}(t)|+\Big|\int_{[0,\,t]}T(t-y){\rm d}V_{k-1}(y)\Big|\\
&\leq \tilde{A}\Big((k-2)\sum_{i=0}^{k-3}\binom{k-3}{i}\frac{c^{k-3-i}t^{i+1}}{(i+1)!{\tt m}^{i+1}}+(k-2)\sum_{i=0}^{k-3}\binom{k-3}{i}\frac{c^{k-2-i}t^i}{i!{\tt m}^i}+V_{k-2}(t)\Big)\\
&\leq \tilde{A}\Big((k-2)\frac{t^{k-2}}{(k-2)!{\tt m}^{k-2}}+(k-2)\sum_{i=1}^{k-3}\binom{k-3}{i-1}\frac{c^{k-2-i}t^i}{i!{\tt m}^i}+(k-2)c^{k-2}\\
&+(k-2)\sum_{i=1}^{k-3}\binom{k-3}{i-1}\frac{c^{k-2-i}t^i}{i!{\tt m}^i}+c^{k-2}+\frac{t^{k-2}}{(k-2)!{\tt m}^{k-2}}+\sum_{i=1}^{k-3}\binom{k-2}{i}\frac{c^{k-2-i}t^i}{i!{\tt m}^i}\Big)\\
&=\tilde{A}(k-1)\Big(c^{k-2}+\frac{t^{k-2}}{(k-2)!{\tt m}^{k-2}}+\sum_{i=1}^{k-3}\binom{k-2}{i}\frac{c^{k-2-i}t^i}{i!{\tt m}^i}\Big)=\tilde{A}(k-1)\sum_{i=0}^{k-2}\binom{k-2}{i}\frac{c^{k-2-i}t^i}{i!{\tt m}^i},
\end{align*}
thereby completing the induction. We have used $\gamma\leq c$ (the limit of a function does not exceed its supremum) and \eqref{inter3} for the first inequality, \eqref{ineq2} for the second inequality, and
\begin{equation}\label{inter6}
\binom{k-3}{i-1}+\binom{k-3}{i}=\binom{k-2}{i}
\end{equation}
for the penultimate equality.

It remains to check \eqref{inter5}. Using once again \eqref{inter6} with $k=j+2$ we obtain, for $j\in\mn$ and $t\geq 0$,
\begin{multline*}
{\tt m}^{-1}\int_0^t W_{j-1}(y){\rm d}y+\gamma W_{j-1}(t)=\frac{t^j}{j!{\tt m}^j}+\sum_{i=1}^{j-1} \binom{j-1}{i-1}\frac{\gamma^{j-i}t^i}{i!{\tt m}^i}+\gamma^j+\sum_{i=1}^{j-1} \binom{j-1}{i}\frac{\gamma^{j-i}t^i}{i!{\tt m}^i}\\=\gamma^j+\sum_{i=1}^{j-1} \binom{j}{i}\frac{\gamma^{j-i}t^i}{i!{\tt m}^i}+
\frac{t^j}{j!{\tt m}^j}=W_j(t).
\end{multline*}
\end{proof}

\begin{proof}[Proof of Proposition \ref{precise}]
Assume that the distribution of $(\xi,\eta)$ satisfies the assumptions of Proposition \ref{precise2}. Then, by \eqref{prec},
$$\frac{j^{1/2}(j-1)!}{{\tt m}^{-j} t^{j-1/2}}\Big(V_j(t)-\frac{t^j}{j!{\tt m}^j}\Big)~\sim~ \frac{{\tt m}\gamma j^{3/2}}{t^{1/2}},\quad t\to\infty$$ whenever $j=j(t)=o(t^{1/2})$ as $t\to\infty$. Thus, the limit of the left-hand side is equal to ${\tt m}\gamma\alpha^{3/2}$ provided that $j=j(t)\sim \alpha t^{1/3}$ as $t\to\infty$, whereas the limit is $+\infty$ provided that $\lim_{t\to\infty}t^{-1/3}j(t)=\infty$.

Now Proposition \ref{precise} follows by setting $(\xi,\eta)=(|\log W|, |\log(1-W)|)$ which results in $V_j(\log n)=\me \rho_j(n)$ and ${\tt m}=\mu$ and then replacing $t$ with $\log n$ and choosing any positive function $t\mapsto j(t)$ that satisfies $j(\log n)=j_n$.
\end{proof}

\section{Appendix}

In this section we shall prove a counterpart for perturbed random walks of a `light' version of the key renewal theorem which applies both in the nonlattice and lattice cases.

Put $V(x):=0$ for $x<0$. We start by noting that
\begin{equation}\label{subad}
V(x+y)-V(x)\leq U(y),\quad x,y\in\mr.
\end{equation}
Indeed, for $x,y\geq 0$,
\begin{eqnarray}\label{eqV}
V(x+y)-V(x)&=&\me (U(x+y-\eta)-U(x-\eta))\1_{\{\eta\leq x\}}+\me U(x+y-\eta)\1_{\{x<\eta\leq x+y\}}\\&\leq& U(y)(\mmp\{\eta\leq x\}+\mmp\{x<\eta\leq x+y\})\leq U(y)\notag
\end{eqnarray}
having utilized subadditivity and monotonicity of $U$ for the penultimate inequality. If $x,y<0$, then both sides of \eqref{subad} are zero. Finally, we use monotonicity of $V$ to obtain: if $x<0$ and $y\geq 0$, then $V(x+y)-V(x)=V(x+y)\leq V(y)\leq U(y)$; and if $x\geq 0$ and $y<0$, then $V(x+y)-V(x)\leq 0=U(y)$.
\begin{lemma}\label{key}
Let $f: \mr\to [0,\infty)$ be a dRi function on $\mr$. Then for some $r>0$ and all $x\in\mr$
\begin{equation}\label{bound}
\int_{[0,\,\infty)} f(x-y){\rm d}V(y)\leq r.
\end{equation}

If $f$ is dRi on $[0,\infty)$ or $(-\infty, 0]$, then the range of integration $[0,\,\infty)$ should be replaced with $[0,\,x]$ or $[x,\infty)$ and then \eqref{bound} holds for all $x\geq 0$ or all $x\leq 0$, respectively.
\end{lemma}
\begin{proof}
We only prove the claim under the assumption that $f$ is dRi on $\mr$. By virtue of the obvious inequality $$f(x)\leq \sum_{n\in\mathbb{Z}}\sup_{y\in [n-1, n)}f(y)\1_{[n-1,\,n)}(x),\quad x\in\mr$$ we obtain
\begin{multline*}
\int_{[0,\,\infty)} f(x-y){\rm d}V(y)\leq \int_{[0,\,\infty)}\sum_{n\in\mathbb{Z}}\sup_{y\in [n-1,\,n)}f(y)\1_{[n-1,\,n)}(x-y){\rm d}V(y)\\=
\sum_{n\in\mathbb{Z}}\sup_{y\in [n-1,\,n)}f(y)(V(x-n+1)-V(x-n))\leq  U(1) \sum_{n\in\mathbb{Z}}\sup_{y\in [n-1,\,n)}f(y)<\infty.
\end{multline*}
We have used \eqref{subad} for the penultimate inequality. Here, the series converges because $f$ is a dRi function on $\mr$.
\end{proof}

\begin{lemma}\label{eq:Borov}
Assume that the distribution of $\xi$ has an absolutely continuous component, $\me e^{\beta_1\xi}<\infty$ and $\me e^{\beta_2\eta}<\infty$ for some $\beta_1,\beta_2>0$. Then
\begin{equation}\label{eq:relat}
V(t)=\frac{t}{{\tt m}}+\Big(\frac{\me\xi^2}{2{\tt m}^{2}}-\frac{\me\eta}{{\tt m}}\Big)+T(t),\quad t\geq 0
\end{equation}
where ${\tt m}=\me\xi<\infty$ and $T$ satisfies
\begin{equation}\label{eq:relat2}
|T(t)|\leq C_0e^{-\beta_0 t},\quad t\geq 0
\end{equation}
for some $C_0>0$ and $\beta_0\in (0,\min(\beta_1,\beta_2))$. 
\end{lemma}
\begin{proof}
Put $a:=\me\xi^2/(2{\tt m}^2)$. According to Theorem 3 on p.~246 in \cite{Borovkov:1976},
\begin{equation*}\label{eq:relat1}
U(t)={\tt m}^{-1} t+a+R(t),\quad t\geq 0,
\end{equation*}
where $|R(t)|\leq Ce^{-\beta t}$, $t\geq 0$ for some $C>0$ and $\beta\in (0, \beta_1)$. Using this 
we obtain
\begin{multline*}
V(t)=\int_{[0,\,t]}U(t-y){\rm d}G(y)=\int_{[0,\,t]}({\tt m}^{-1}(t-y)+a+R(t-y)){\rm d}G(y)={\tt m}^{-1}t+a-{\tt m}^{-1}\me\eta\\+{\tt m}^{-1}\int_t^\infty(1-G(y)){\rm d}y-a(1-G(t))+\int_{[0,\,t]}R(t-y){\rm d}G(y),
\end{multline*}
that is, \eqref{eq:relat} holds with $$T(t):={\tt m}^{-1}\int_t^\infty(1-G(y)){\rm d}y-a(1-G(t))+\int_{[0,\,t]}R(t-y){\rm d}G(y).$$ It remains to show that $T$ is exponentially bounded. To this end, we first note that manipulating with $\beta$ and $C$ we can ensure that $\beta\in (0,\beta_2)$. Then $\me e^{\beta \eta}<\infty$ and $1-G(t)\leq C_1 e^{-\beta t}$, $t\geq 0$ for some $C_1>0$. With this at hand, $$T(t)\leq {\tt m}^{-1}\int_t^\infty(1-G(y)){\rm d}y+\int_{[0,\,t]}R(t-y){\rm d}G(y)\leq (({\tt m}\beta)^{-1}C_1+C\me e^{\beta\eta})e^{-\beta t},\quad t\geq 0$$ and
$$T(t)\geq -a(1-G(t))+\int_{[0,\,t]}R(t-y){\rm d}G(y)\geq -(C_1a+C\me e^{-\beta\eta})e^{-\beta t},\quad t\geq 0.$$ Thus, \eqref{eq:relat2} holds with $\beta_0:=\beta$ and $C_0:=C\me e^{-\beta\eta}+C_1\max (({\tt m}\beta)^{-1}, a)$.
\end{proof}

\vspace{5mm}

\noindent {\bf Acknowledgement}. The present work was supported by National Research Foundation of Ukraine (project 2020.02/0014 'Asymptotic regimes of perturbed random walks: on the edge of modern and classical probability').


\begin{thebibliography}{99}

\footnotesize

\bibitem{Alsmeyer+Iksanov+Marynych:2017} G. Alsmeyer, A. Iksanov and A. Marynych, \textit{Functional limit theorems for the number of occupied boxes in the Bernoulli sieve}. Stoch. Proc. Appl. \textbf{127} (2017), 995--1017.


\bibitem{Bertoin:2008} J. Bertoin, \textit{Asymptotic regimes for the occupancy scheme of multiplicative cascades}. Stoch. Proc. Appl. \textbf{118} (2008), 1586--1605.

\bibitem{Borovkov:1976} A.~A. Borovkov, \textit{Stochastic processes in queuing theory}. Springer, 1976.

\bibitem{Buraczewski+Dovgay+Iksanov:2020} D. Buraczewski, B. Dovgay and A. Iksanov, \textit{On intermediate levels of nested occupancy scheme in random environment generated by stick-breaking I}. Electron. J. Probab. \textbf{25} (2020), paper no. 123, 24 pp.

\bibitem{Businger:2017} S. Businger, \textit{Asymptotics of the occupancy scheme in a random environment and its applications to tries}. Discrete Mathematics and Theoretical Computer Science \textbf{19} (2017), \#22.

\bibitem{Duchamps+Pitman+Tang:2017+} J.-J. Duchamps, J. Pitman and W. Tang, \textit{Renewal sequences and record chains related to multiple zeta sums}. Trans. Amer. Math. Soc. \textbf{371} (2019), 5731--5755.

\bibitem{Gnedin:2004} A.~V. Gnedin, \textit{The Bernoulli sieve}. Bernoulli \textbf{10} (2004), 79--96.

\bibitem{Gnedin+Iksanov:2020} A. Gnedin and A. Iksanov, \textit{On nested infinite occupancy scheme in random environment}. Probab. Theory Relat. Fields \textbf{177} (2020), 855--890.

\bibitem{Iksanov:2013} A. Iksanov, \textit{Functional limit theorems for renewal shot noise
processes with increasing response functions}. Stoch. Proc. Appl. \textbf{123} (2013), 1987--2010.

\bibitem{Iksanov:2016} A. Iksanov, \textit{Renewal theory for perturbed random walks and similar processes}. Birkh\"{a}user, 2016.

\bibitem{Iksanov+Jedidi+Bouzeffour:2017} A. Iksanov, W. Jedidi and F. Bouzeffour, \textit{A law of the iterated logarithm for the number of occupied boxes in the Bernoulli sieve}. Statist. Probab. Letters \textbf{126} (2017), 244--252.

\bibitem{Joseph:2011} A. Joseph, \textit{A phase transition for the heights of a fragmentation
tree}. Random Structures and Algorithms \textbf{39} (2011),
247--274.

\bibitem{Karlin:1967} S. Karlin, \textit{Central limit theorems for certain infinite
urn schemes}. J. Math. Mech. \textbf{17} (1967), 373--401.


\bibitem{Pitman+Tang:2017+} J. Pitman and W. Tang, \textit{Regenerative random permutations of integers}. Ann. Probab.
\textbf{47} (2019), 1378--1416.

\end{thebibliography}
\end{document}